\documentclass{amsart}
\usepackage{amssymb}
\usepackage{amsmath}
\usepackage{textcomp}
\makeindex

\newcommand{\ba}{\begin{array}}
\newcommand{\eea}{\end{eqnarray}}
\newcommand{\ea}{\end{array}}

\newtheorem{definition}{Definition}[section]
\newtheorem{theorem}[definition]{Theorem}

\newtheorem{remark}[definition]{Remark}

\usepackage[matrix,arrow,curve]{xy}
\usepackage[utf8x]{inputenc}
\usepackage{tikz}

\begin{document}
\title[Poisson Structures On Closed Manifolds]{Poisson Structures On Closed Manifolds}
\author[S. Mukherjee]{Sauvik Mukherjee}
\address{Presidency University, Kolkata, India.\\
e-mail:mukherjeesauvik@yahoo.com\\}

\begin{abstract} We prove an $h$-principle for poisson structures on closed manifolds. Equivalently we prove $h$-principle for symplectic foliation (singular) on closed manifolds. On open manifolds however the singularities could be avoided and it is a known result by Fernandes and Frejlich \cite{Fernandes}. \end{abstract}
\maketitle

\section{introduction} In this paper we prove an $h$-principle for poisson structures on closed manifolds. Similar results on open manifolds has been proved by Fernandes and Frejlich in \cite{Fernandes}. We state their result below.\\

 Let $M^{2n+q}$ be a $C^{\infty}$-manifold equipped with a co-dimension-$q$ foliation $\mathcal{F}_0$ and a $2$-form $\omega_0$ such that $(\omega_0^n)_{\mid T\mathcal{F}_0}\neq 0$. Denote by $Fol_q(M)$ the space of co-dimension-$q$ foliations on $M$ identified as a subspace of $\Gamma(Gr_{2n}(M))$, where $Gr_{2n}(M)\stackrel{pr}{\to} M$ is the grassmann bundle, i.e, $pr^{-1}(x)=Gr_{2n}(T_xM)$ and $\Gamma(Gr_{2n}(M))$ is the space of sections of $Gr_{2n}(M)\stackrel{pr}{\to} M$ with compact open topology. Define \[\Delta_q(M)\subset Fol_q(M)\times \Omega^2(M)\] \[\Delta_q(M):=\{(\mathcal{F},\omega):\omega^n_{\mid T\mathcal{F}}\}\neq 0\] Obviously $(\mathcal{F}_0,\omega_0)\in \Delta_q(M)$. In this setting Fernandes and Frejlich has proved the following 

\begin{theorem}(\cite{Fernandes})
\label{Fernandes}
Let $M^{2n+q}$ be an open manifold with $(\mathcal{F}_0,\omega_0)\in \Delta_q(M)$ be given. Then there exists a homotopy $(\mathcal{F}_t,\omega_t)\in \Delta_q(M)$ such that $\omega_1$ is $d_{\mathcal{F}_1}$-closed (actually exact). 
\end{theorem}

In the language of poisson geometry the above result \ref{Fernandes} takes the following form. Let $\pi \in \Gamma(\wedge ^2TM)$ be a bi-vectorfield on $M$, define $\#\pi:T^*M\to TM$ as $\#\pi(\eta)=\pi(\eta,-)$. If $Im(\#\pi)$ is a regular distribution then $\pi$ is called a regular bi-vectorfield.

 \begin{theorem}
 \label{Fernandes-1}
 Let $M^{2n+q}$ be an open manifold with a regular bi-vectorfield $\pi_0$ on it such that $Im(\#\pi)$ is an integrable distribution then $\pi_0$ can be homotoped through such bi-vectorfields to a poisson bi-vectorfield $\pi_1$.
 \end{theorem}

In \ref{Fernandes} above $d_{\mathcal{F}}$ is the tangential exterior derivative, i.e, for $\eta \in \Gamma(\wedge^k T^*\mathcal{F})$, $d_{\mathcal{F}}\eta$ is defined by the following formula \[d_{\mathcal{F}}\eta(X_0,X_1,...,X_k)=\Sigma_i(-1)^iX_i(\eta(X_0,..,\hat{X}_i,..,X_k))\]\[+\Sigma_{i<j}(-1)^{i+j}\eta([X_i,X_j],X_0,..,\hat{X}_i,..,\hat{X}_j,..,X_k)\]where $X_i\in \Gamma(T\mathcal{F})$. So if we extend a $\mathcal{F}$-leafwise closed $k$-form $\eta$, i.e, $d_{\mathcal{F}}\eta=0$, to a form $\eta'$ by the requirement that $ker(\eta')=\nu \mathcal{F}$, where $\nu \mathcal{F}$ is the normal bundle to $\mathcal{F}$, then $d\eta'=0$.\\ 

In order to fix the foliation in \ref{Fernandes} the foliated manifold $(M,\mathcal{F})$ must be uniformly open. Let us define this notion.

\begin{definition}(\cite{Bertelson})
\label{Uniform open}
A foliated manifold $(M,\mathcal{F})$ is called uniformly open if there exists a function $f:M\to [0,\infty)$ such that 
\begin{enumerate}
\item $f$ is proper,\\
\item $f$ has no leafwise local maxima,\\
\item $f$ is $\mathcal{F}$-generic.
\end{enumerate}
\end{definition} 

So let us explain the notion $\mathcal{F}$-generic. In order to do so we need to define the singularity set $\Sigma^{(i_1,i_2,...,i_k)}(f)$ for a map $f:M\to W$. $\Sigma^{i_1}(f)$ is the set \[\{p\in M:dim(ker(df)_p)=i_1\}\] It was proved by Thom \cite{Thom} that for most maps $\Sigma^{i_1}(f)$ is a submanifold of $M$. So we can restrict $f$ to $\Sigma^{i_1}(f)$ and construct $\Sigma^{(i_1,i_2)}(f)$ and so on. In \cite{Thom} it has been proved that there exists $\Sigma^{(i_1,...,i_k)}\subset J^k(M,W)$ such that $(j^kf)^{-1}\Sigma^{(i_1,...,i_k)}=\Sigma^{(i_1,...,i_k)}(f)$.\\

 Let us set $W=\mathbb{R}$ as this is the only situation we need. Let $(M,\mathcal{F})$ be a foliated manifold with a leaf $F$. Define the restriction map \[r_F:J^k(M,\mathbb{R})\to J^k(F,\mathbb{R}):j^kf(x)\mapsto j^k(f_{\mid F})(x)\] Define foliated analogue of the singularity set as \[\Sigma^{(i_1,i_2,...,i_k)}_{\mathcal{F}}:=\cup_{\{F\ leaf\ of\ \mathcal{F}\}} r_F^{-1}\Sigma^{(i_1,i_2,...,i_k)}\]

\begin{definition}(\cite{Bertelson})
A smooth real valued function $f:M\to \mathbb{R}$ is called $\mathcal{F}$-generic if the first jet $j^1f \pitchfork \Sigma^{(n)}_{\mathcal{F}}$ and the second jet $j^2f \pitchfork \Sigma^{(i_1,i_2)}_{\mathcal{F}}$ for all $(i_1,i_2)$.
\end{definition}

 We refer the reader to \cite{Bertelson} for more details. Under this hypothesis Bertelson proved the following

\begin{theorem}(\cite{Bertelson})
If $M$ is open and $(M,\mathcal{F})$ be a uniformly open foliated manifold and let $\omega_0$ be a $\mathcal{F}$-leafwise non-degenerate $2$-form then $\omega_0$ can be homotoped through $\mathcal{F}$-leafwise non-degenerate $2$-forms to a $\mathcal{F}$-leafwise symplectic form. 
\end{theorem}

She also constructed counter examples in \cite{Bertelson1} that without these conditions the above theorem fails. A contact analogue of Bertelson's result on any manifold (open or closed) has recently been proved in \cite{Over-twisted} by Borman, Eliashberg and Murphy. We will use this theorem in our argument. So let us state the theorem.

\begin{theorem}(\cite{Over-twisted})
\label{Contact}
Let $M^{2n+q+1}$ be any manifold equipped with a co-dimension-$q$ foliation $\mathcal{F}$ on it and let $(\alpha_0,\beta_0)\in \Gamma(T^*\mathcal{F}\oplus \wedge^2T^*\mathcal{F})$ be given such that $\alpha_0\wedge \beta_0^n$ is nowhere vanishing, then there exists a homotopy $(\alpha_t,\beta_t)\in \Gamma(T^*\mathcal{F}\oplus \wedge^2T^*\mathcal{F})$ such that $\alpha_t\wedge \beta_t^n$ is nowhere vanishing and $\beta_1=d_{\mathcal{F}}\alpha_1$.
\end{theorem}

\begin{definition}
\label{Homotopy of singular foliation}
By a homotopy of singular foliation $\mathcal{F}_t,\ t\in I$ on a manifold $M$ we mean a regular foliation  $\mathcal{F}$ on $M\times I$. The singular locus $\Sigma_t$ of $\mathcal{F}_t$ is given by \[\Sigma_t=\{(x,t)\in M\times\{t\}:\mathcal{F}_{(x,t)}\ not\ \pitchfork\ to\ M\times \{t\}\}\]   
\end{definition}

Now we state the main theorem of this paper. 

\begin{theorem}
\label{Main}
Let $M^{2n+q}$ be a closed manifold with $q=2$ and $(\mathcal{F}_0,\omega_0)\in \Delta_q(M)$ be given. Then there exists a homotopy $\mathcal{F}_t$ of singular foliations on $M$ with singular locus $\Sigma_t$ and a homotopy of two forms $\omega_t$ such that the restriction of $(\omega_t)$ to $T\mathcal{F}_t$ is non-degenerate and $\omega_1$ is $d_{\mathcal{F}_1}$-closed, i.e, $(\mathcal{F}_1,\omega_1)$ is a symplectic foliation (singular).
\end{theorem}

It is known from Theorem:1.5.6 of \cite{Dufour} and Theorem:2.14 of \cite{Vaisman} that a symplectic foliation determines a poisson structure. Moreover any foliation (singular or regular) with a leafwise non-degenerate $2$-form determines a bivectorfield. Hence in terms of poisson geometry \ref{Main} states 

\begin{theorem}
\label{Main-1}
Let $M^{2n+q}$ be a closed manifold with $q=2$ and $\pi_0$ be a regular bi-vectorfield of rank $2n$ on it such $Im(\#\pi_0)$ is integrable distribution. Then there exists a homotopy of bi-vectorfields $\pi_t,\ t\in I$ (not regular) such that $Im(\#\pi_t)$ is integrable and $\pi_1$ is a poisson bi-vectorfield.
\end{theorem}

We organize the paper as follows. In section-2 we shall explain the preliminaries of the theory of $h$-principle and of wrinkle maps which are needed in the proof of \ref{Main} which we present in section-3. 

\section{Preliminaries} We begin with the theory of $h$-principle. Let $X\to M$ be any fiber bundle and let $X^{(r)}$ be the space of $r$-jets of jerms of sections of $X\to M$ and $j^rf:M\to X^{(r)}$ be the $r$-jet extension map of the section $f:M\to X$. A section $F:M\to X^{(r)}$ is called holonomic if there exists a section $f:M\to X$ such that $F=j^rf$. In the following we use the notation $Op(A)$ to denote a small open neighborhood of $A\subset M$ which is unspecified.   

\begin{theorem}(\cite{Eliashberg} Holonomic Approximation Theorem)
\label{HAT}
Let $A\subset M$ be a polyhedron of positive co-dimension and $F_z:Op(A)\to X^{(r)}$ be a family of sections parametrized by a cube $I^m,\ m=0,1,2,...$ such that $F_z$ is holonomic for $z\in Op(\partial I^m)$. Then for given small $\varepsilon,\delta >0$ there exists a family of $\delta$-small (in the $C^0$-sense) diffeotopies $h^{\tau}_{z}:M\to M,\ \tau\in [0,1],\ z\in I^m$ and a family of holonomic sections $\tilde{F}_z:Op(h^1_z(A))\to X^{(r)},\ z\in I^m$ such that 
\begin{enumerate}
\item $h^{\tau}_z=id_M$ and $\tilde{F}_z=F_z$ for all $z\in Op(\partial I^m)$\\
\item $dist(\tilde{F}_z(x),(F_z)_{\mid Oph^1_z(A)}(x))<\varepsilon$ for all $x\in Op(h^1_z(A))$\\ 
\end{enumerate} 

\end{theorem}

\begin{remark}
\label{Relative}
Relative version of \ref{HAT} is also true. More precisely let the sections $F_z$ be already holonomic on $Op(B)$ for a sub-polyhedron $B$ of $A$, then the diffeotopies $h^{\tau}_z$ can be made to be fixed on $Op(B)$ and $\tilde{F}_z=F_z$ on $Op(B)$.  
\end{remark}

Let $\mathcal{R}$ be a subset of $X^{(r)}$. Then $\mathcal{R}$ is called a differential relation of order $r$. $\mathcal{R}$ is said to satisfy $h$-principle if any section $F:M\to \mathcal{R}\subset X^{(r)}$ can be homotopped to a holonomic section $\tilde{F}:M\to \mathcal{R}\subset X^{(r)}$ through sections whose images are contained in $\mathcal{R}$.\\

Now we briefly recall preliminaries of wrinkled maps following \cite{Wrinkle}. Consider the following maps \[w,e:\mathbb{R}^{q-1}\times \mathbb{R}^{2n}\times \mathbb{R}\to \mathbb{R}^{q-1}\times \mathbb{R}\] \[w_s(y,x,z)=(y,z^3+3(|y|^2-1)z-\Sigma_1^sx_i^2+\Sigma_{s+1}^{2n}x_i^2)\]
 \[e_s(y,x,z)=(y,z^3+3|y|^2z-\Sigma_1^sx_i^2+\Sigma_{s+1}^{2n}x_i^2)\] where $y\in \mathbb{R}^{q-1}$, $z\in \mathbb{R}$ and $x\in \mathbb{R}^{2n}$. Observe that the singular locus of $w_s$ is \[\Sigma(w_s)=\{x=0,\ z^2+|y|^2=1\}\] Let $D$ be the disc enclosed by $\Sigma(w_s)$, i.e, \[D=\{x=0,\ z^2+|y|^2\leq 1\}\] 

A fibered map over $B$ is given by a map $f:U\to V$, where $U\subset M$ and $V\subset Q$ with submersions $a:U\to B$ and $b:V\to B$ such that $b\circ f=a$.\\
Denote by $T_BM$ and $T_BQ$, the distributions i.e, subbundles $ker(a)\subset TM$ and $ker(b)\subset TQ$ respectively. The fibered differential $df_{\mid T_BM}$ is denoted by $d_Bf$. A fibered submersion is a map $f$ as above such that $d_Bf$ is fiberwise surjective. Similarly we can define fibered epimorphism $\phi:T_BM\to T_BQ$ and a homotopy of fibered epimorphism is family of fibered epimorphisms $\phi_t$.\\

Now we define the equivalence of fibered maps. So consider two fibered maps $f:U\to V$ and $f':U'\to V'$ over the base $B$ and $B'$ with submersions $a,b$ and $a',b'$ respectively. $f$ and $f'$ are called equivalent if there exist open subsets $A\subset B$, $A'\subset B'$, $W\subset V$, $W'\subset V'$ with diffeomorphisms $\phi:U\to U'$, $\psi:W\to W'$ and $s:A\to A'$ such that the following diagram commutes. \\

\[
\xymatrix@=2pc@R=2pc{
& U\ar@{->}[rrrr]^f \ar@{->}[dr]^{\phi}\ar@{->}[dddrr]_a & & & & W \ar@{->}[dl]_{\psi}\ar@{->}[dddll]^b\\
& & U'\ar@{->}[rr]^{f'}\ar@{->}[dr]_{a'} & & W'\ar@{->}[dl]^{b'} &\\
& &  & A' & &\\
& &  & A\ar@{->}[u]^s & &\\ 
}
\]

Observe that if we consider the projection on first $k$-factors, where $k<q-1$, then $w_s$ is a fibered map.\\

\begin{definition}(\cite{Wrinkle})
\label{Wrinkled map}
A fibered map $f:M^{2m+q}\to Q^q$ between smooth manifolds is called a fibered wrinkled map if there exists a disjoint union of open subsets $U_1,...,U_l\subset M$ such that $f_{\mid M-U}$ is a fibered submersion, where $U=\cup_1^lU_i$ and $f_{\mid U_i}$ is equivalent to $w_s$, for some $s$.\\
It is called an embryo if $f_{\mid U_i}$ is equivalent to $e_s$, for some $s$.
\end{definition}

A special wrinkle map is a map $\hat{\alpha}:\Omega \times \mathbb{R}\to \mathbb{R}^{n},\ (\Omega\subset \mathbb{R}^{n-1}-open)$ given by \[\hat{\alpha}(x_1,...,x_n)=(x_1,...,x_{n-1},\alpha(x_1,...,x_n))\] where $\alpha$ is given by \[\alpha(x_1,...,x_n)=x_n^3-3\mu_1(x_1,...,x_{n-1})x_n+\mu_2(x_1,...,x_{n-1})\] where $\mu_1,\mu_2$ are arbitrary functions. The singularity set $\Sigma$ of $\hat{\alpha}$ is given by \[\Sigma=\{(x_1,...,x_{n-1}):x_n^2=\mu_1(x_1,...,x_{n-1})\}\]

 We refer to \cite{Wrinkle} for more details on special wrinkle.\\

Let $p:\mathbb{R}^{2n+q}\to \mathbb{R}^{2n+q-1}$ be the projection on the first $2n+q-1$ factors and let $\beta:[a,b]\to \mathbb{R}$ be a 1-dimensional wrinkle i.e, $\beta$ is a Morse function with two critical points $c$ and $d$ respectively with $c$ the maximum and $d$ the minimum. Then $\beta(c)-\beta(d)$ is called the span of the wrinkle $\beta$. If $f=\hat{\alpha}$ is a special wrinkle then for each $x'\in \Omega$ the restriction $\alpha_{\mid p^{-1}(x')}$ is either non-singular, a wrinkle or an embryo. The function $s_f(x')=span(\alpha_{\mid p^{-1}(x')})$ is called the span function of the wrinkle $f$.\\

The base of $f$ is $p(\Sigma)$. A wrinkle is called small if both its base and span is small.\\ 

 We refer the reader \cite{Wrinkle} for more details. By combining Lemma-2.1B and Lemma-2.2B of \cite{Wrinkle} we get the following

\begin{theorem}(\cite{Wrinkle})
\label{WrinkleKey}
Let $g:I^n\to I^q$ be a fibered submersion over $I^k$ and $\theta:I^n\to I^n$ be a fibered wrinkled map over $I^k$ with one wrinkle. Then there exists a fibered wrinkled map $\psi$ with very small wrinkles and which agrees with $\theta$ near $\partial I^n$ such that $g\circ\psi$ is a fibered wrinkled map.
\end{theorem}

For the definition of $\psi$ in the conclusion of \ref{WrinkleKey} above please see the Appendix \ref{Appendix}. But it does not contain a proof for which we refer \cite{Wrinkle}.

\section{Main Theorem} In this section we prove \ref{Main}. \\

Consider $\tilde{M}=M\times \mathbb{R}$ and let us denote the co-dimension-$q$ foliation $\mathcal{F}_0\times \mathbb{R}$ on $\tilde{M}$ by $\tilde{\mathcal{F}}$ with a $\tilde{\mathcal{F}}$-leafwise one form $\alpha_0$ such that $\alpha_0(\partial_s)=1$ and $ker(\alpha_0)_{\mid (x,s)}=T_x\mathcal{F}_0$. Observe that if we extend $\omega_0$ to $\tilde{M}$ by the requirement that $\omega_0(\partial_s,-)=0$, then $(\alpha_0\wedge \omega_0^n)_{\mid T\tilde{\mathcal{F}}}\neq 0$. Let $(\omega_0)_{\mid T\tilde{\mathcal{F}}}=\beta_0$. Then $(\alpha_0,\beta_0)$ is a $\tilde{\mathcal{F}}$-leafwise almost contact structure. Then according \ref{Contact} there exists a homotopy of pairs $(\alpha_t,\beta_t)$ defining a homotopy of $\tilde{\mathcal{F}}$-leafwise almost contact structures consisting of a $\tilde{\mathcal{F}}$-leafwise one form $\alpha_t$ and a $\tilde{\mathcal{F}}$-leafwise two form $\beta_t$ such that $\beta_1=d_{\tilde{\mathcal{F}}}\alpha_1$, i.e, $(\alpha_1,\beta_1)$ is a $\tilde{\mathcal{F}}$-leafwise contact structure. Now let $L_t=ker(\alpha_t)\subset T\tilde{\mathcal{F}}$ and $G_t^1=L_t\oplus \nu\tilde{\mathcal{F}}\oplus \mathbb{R}$, where $\nu\tilde{\mathcal{F}}$ is the normal bundle.\\

Now observe that the embedding $f_0:M\to M\times \{0\}\hookrightarrow \tilde{M}\times \mathbb{R}$ is $\pitchfork$ to $\tilde{\mathcal{F}}\times \mathbb{R}$ and $Im(df_0)\cap (T\tilde{\mathcal{F}}\times \mathbb{R})=L_0$. First extend $\beta_t$ to $\tilde{M}$ and call it $\tilde{\beta}_t$ in such a way that $ker(\tilde{\beta}_t)=\nu \tilde{\mathcal{F}}$. Let $X_t=ker(\beta_t)$ be the vector field on $\tilde{M}$ and consider the family of $2$-dimensional foliation $\mathcal{G}_t$ generated by $X_t$ and $\partial_w$, where $w$ is the $\mathbb{R}$-variable in $\tilde{M}\times \mathbb{R}$. Observe that $\alpha_t\wedge dw$ is a  $\mathcal{G}_t$-leafwise symplectic form.\\

 Now we shall perturb $f_0$ by a homotopy of immersions $f_t$ such that $f_t$ will be tangent to $\tilde{\mathcal{F}}\times \mathbb{R}$ only on $\Sigma_t$ and on $M-\Sigma_t$, $f_t\pitchfork \tilde{\mathcal{F}}\times \mathbb{R}$ (observe that non-transversallity does not mean tangency in this context), i.e, $Im(df_t)\cap (T\tilde{\mathcal{F}}\times \mathbb{R})$ is of dimension $2n$ and $Im(df_t)\cap (T\tilde{\mathcal{F}}\times \mathbb{R})$ is close to $L_t$. As $\tilde{\beta_t}^n_{\mid L_t}\neq 0$, we conclude that the restriction of $\tilde{\beta_t}+ \alpha_t\wedge dw$ is non-degenerate on $Im(df_t)\cap T\tilde{\mathcal{F}}\times \mathbb{R}$. Hence we only need to set $\mathcal{F}_t=f_t^{-1}(\tilde{\mathcal{F}}\times \mathbb{R})$ and $\omega_t=f_t^*e^w(\tilde{\beta}_t+\alpha_t\wedge dw)$.\\

First divide the interval $I$ as \[I=\cup_1^N[(i-1)/N,i/N]\] and assume that $f_t$ is defined on $[0,(i-1)/N]$. Observe that the limit \[lim_{x\to \Sigma_{(i-1)/N}}Im(df_{(i-1)/N})\cap (T\tilde{\mathcal{F}}\times \mathbb{R})\] exists and is of dimension $2n$ and is close to $L_{(i-1)/N}$. Let $\bar{L}_{(i-1)/N}\subset T\tilde{\mathcal{F}}\times \mathbb{R}$ be the $2n$-dimensional distribution which equals $Im(df_{(i-1)/N})\cap T\tilde{\mathcal{F}}\times \mathbb{R}$ on $M-\Sigma_{(i-1)/N}$ and on $\Sigma_{(i-1)/N}$ it is the limit. Set $\nu_{(i-1)/N}=Im(df_{(i-1)/N})/\bar{L}_{(i-1)/N}$ and $G_t^i,\ t\in [(i-1)/N,i/N]$ as \[G_t^i=L_t\oplus \nu_{(i-1)/N}\] Observe that $Im(df_{(i-1)/N})$ approximates $G_{(i-1)/N}^i$. So if $N$ is large then there exists a family of monomorphisms $F_t,\ t\in[(i-1)/N,i/N]$ such that $F_{(i-1)/N}=df_{(i-1)/N}$ and $Im(F_t)$ approximates $G_t^i$ and hence $F_t$ tangent to $T\tilde{F}\times \mathbb{R}$ only on a slightly perturbed $\Sigma_{(i-1)/N}$. \\

Choose a triangulation of $M$ which is fine and $\Sigma_{(i-1)/N}\subset A$, where $A$ is the $(2n+q-1)$-skeleton of the triangulation. As the triangulation is fine all $(2n+q)$-simplices under the image of $f_{(i-1)/N}$ is contained in a neighborhood diffeomorphic to $I^{2n+q+2}$ and on it $\tilde{\mathcal{F}}\times \mathbb{R}$ is given by the projection $\pi:I^{2n+q+2}\to I^q$ (projection on the first $q$-factors). \\

Without loss of generality let us assume $F_t$ is defined for $t\in I$ instead of $t\in [(i-1)/N,i/N]$. Let \[\bar{F}_t=F_{\sigma(t)}\] where $\sigma:I\to I$ is a smooth map such that $\sigma=0$ on $[0,\varepsilon]\cup [1-\varepsilon,1]$ and $\sigma=1$ on a neighborhood of $1/2$. Observe that $\bar{F}_t$ is holonomic for $t\in \partial I$.\\

Use \ref{HAT} for $\bar{F}_t$ to get a family of immersions $\bar{f}_t$ defined on $Op(h^1_t(A))$ so that $d\bar{f}_t$ approximates $\bar{F}_t$ on $Op(h^1_t(A))$, where $h^{\tau}_t$ is $\delta$ small with $h^1_t=id$ for $t\in [0,\varepsilon]\cup [1-\varepsilon,1]$. The $\delta$ above will be used later so the reader needs to keep note of this fact.  We approximate $\bar{F}_t$ by $F'_t$ such that $F'_t=d\bar{f}_t$ on $Op(h^1_t(A))$.\\

It is enough to consider one simplex $\Delta$. Let $\Delta'\subset \Delta$ be a $2\delta$-smaller simplex so that $h^1_{t}(\Delta)$ does not intersect $\Delta'$, $\delta$ is produced by applying \ref{HAT} to $\bar{F}_t$ above. \\

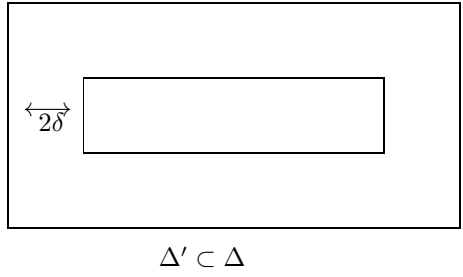
 \begin{center}
\begin{picture}(300,150)(-100,5)\setlength{\unitlength}{1cm}
\linethickness{.075mm}

\multiput(-1,1.5)(6,0){2}
{\line(0,1){3}}
\multiput(0,2.5)(4,0){2}
{\line(0,1){1}}

\multiput(-1,1.5)(0,3){2}
{\line(1,0){6}}
\multiput(0,2.5)(0,1){2}
{\line(1,0){4}}
\put(1,1){$\Delta'\subset \Delta$}
\put(-.8,3){$\longleftrightarrow$}
\put(-.6,2.8){$2\delta$}

\end{picture}\end{center}

Define monomorphisms $\tilde{F}^{\delta}_t$ depending on $\delta$ as follows. On $Op(\partial \Delta)$, set $\tilde{F}^{\delta}_t=d(\bar{f}_t\circ h^1_t)$. Now observe there exists an isotopy of embeddings \[\tilde{g}_{\tau}:\Delta-Op(\partial \Delta)\to \Delta-Op(\partial \Delta)\] such that $\tilde{g}_0=id$ and $\tilde{g}_1(\Delta-Op(\partial \Delta))=\Delta'$. Any element of $(\Delta-Op(\partial \Delta))-\Delta'$ is of the form $\tilde{g}_{\tau}(x),\ x\in \partial(\Delta-Op(\partial \Delta))$ for some value of $\tau$.\\

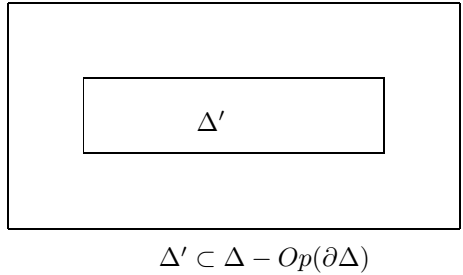
\begin{center}
\begin{picture}(300,150)(-100,5)\setlength{\unitlength}{1cm}
\linethickness{.075mm}

\multiput(-1,1.5)(6,0){2}
{\line(0,1){3}}
\multiput(0,2.5)(4,0){2}
{\line(0,1){1}}

\multiput(-1,1.5)(0,3){2}
{\line(1,0){6}}
\multiput(0,2.5)(0,1){2}
{\line(1,0){4}}
\put(1,1){$\Delta'\subset \Delta-Op(\partial \Delta)$}
\put(1.5,2.8){$\Delta'$}
\end{picture}\end{center}

In the above picture $\Delta'$ is represented by the inner rectangle and the outer rectangle represents $\Delta-Op(\partial \Delta)$. The complement of $\Delta'$ in the above picture shrinks and eventually vanishes by $\tilde{g}_{\tau}$ as $\tau$ varies from $0$ to $1$.\\

Let $\gamma^x_t:I\to M$ be the path 
\[
\begin{array}{rcl}
\gamma^x_t(\tau) &=& h^{1-2\tau}_t(x),\ \tau \in [0,1/2]\\
&=& \tilde{g}_{2\tau-1}(x),\ \tau \in [1/2,1]
\end{array}  
\]
Set $(\tilde{F}^{\delta}_t)_{\tilde{g}_{\tau}(x)}=(F'_t)_{\gamma^x_t(\tau)}$. Observe that $\gamma^x_t(1)=\tilde{g}_1(x)\in \partial \Delta'$. As $\tilde{F}^{\delta}_t$-agrees with $F'_t$ along $\partial \Delta'$, we can extend $\tilde{F}^{\delta}_t$ on $\Delta$ by defining it to be $F'_t$ on $\Delta'$. Observe that \[\Sigma^{\delta}_t=\{\tilde{\mathcal{F}}^{\delta}_t\ tangent\ to\ T\tilde{\mathcal{F}}\times \mathbb{R}\}\subset \Delta-\Delta'\] The next theorem \ref{Key} extends $f_t$ from $t\in [0,(i-1)/N]$ to $t\in [0,i/N]$. To start the process i.e, to extend $f_0$ to $f_t,\ t\in[0,1/N]$ we take a fine triangulation of $M$ so that image under $f_0$ of all top dimensional simplices lies in a neighborhood diffeomorphic to $I^{2n+q+2}$ and on it $\tilde{\mathcal{F}}\times \mathbb{R}$ is given by the projection on the first $q$ factors $\pi:I^{2n+q+2}\to I^q$. 

\begin{theorem}
\label{Key}
Let $I_{\delta}=[\delta,1-\delta]$, $I_{\varepsilon}=[\varepsilon,1-\varepsilon]$ with $\varepsilon=\varepsilon(\delta)<\delta$ and $(F^{\delta}_t,b^{\delta}_t):TI^{2n+q}\to TI^{2n+q+2}$ be a family of monomorphisms such that 
\begin{enumerate}
\item $F^{\delta}_t=db^{\delta}_t$ on $I^{2n+q}-I^{2n+q}_{\varepsilon(\delta)}$, $F^{\delta}_0=db^{\delta}_0\ on\ I^{2n+q}$\\
\item $F^{\delta}_t\pitchfork \mathcal{L}$ on $I^{2n+q}_{\delta}$ for all $t$ and $Im(F^{\delta}_t)\cap T\mathcal{L}$ is of dimension $2n$ and is close to $L_t$ for all $t$ on $I^{2n+q}_{\delta}$\\
\item $\Sigma^{\delta}_t=\{F^{\delta}_t\ tangent\ to\ T\mathcal{L}\}\subset (I^{2n+q}-I^{2n+q}_{\delta})$\\
\end{enumerate}
where $\mathcal{L}$ is the foliation on $I^{2n+q+2}$ induced by the projection $\pi: I^{2n+q+2}\to I^q$ (projection on the first $q$-factors), $\tilde{\mathcal{L}}$ is such that $\mathcal{L}=\tilde{\mathcal{L}}\times I$ and $L_t\subset T\tilde{\mathcal{L}}$ is a family of $2n$-dimensional distribution. Then there is a $\delta''$ and a family of immersions $f_t:I^{2n+q}\to I^{2n+q+2}$ such that 
\begin{enumerate}
\item $f_t=b^{\delta''}_t$ on $I^{2n+q}-I^{2n+q}_{\varepsilon(\delta'')/2}$\\
\item $(\pi\circ f_t)_{\mid I^{2n+q}_{\delta''}}$ is a wrinkle map\\
\item If $\Sigma_t(I^{2n+q}-I^{2n+q}_{\delta''})=\{x\in I^{2n+q}-I^{2n+q}_{\delta''}: f_t(x)\ tangent\ to\ \mathcal{L}\}$, then on $(I^{2n+q}-I^{2n+q}_{\delta''})-\Sigma_t(I^{2n+q}-I^{2n+q}_{\delta''})$, $f_t\pitchfork \mathcal{L}$, $Im(df_t)\cap T\mathcal{L}$ is of dimension $2n$ and is close to $L_t$.\\
\end{enumerate}

Moreover we can modify $f_t$ near the wrinkles such that $f_t$ becomes tangent to $\mathcal{L}$ along the wrinkles.
\end{theorem}

\begin{remark}
Observe that \ref{Key} above completes the induction process and hence the proof of \ref{Main} and \ref{Main-1}.
\end{remark}

\begin{proof}
Let $\sigma:I\to I$ be a smooth map such that $\sigma=0$ on $I-I_{\varepsilon(\delta)}$ and $\sigma=1$ on a neighborhood of $1/2$. Let \[F^{\delta}:T(I\times I^{2n+q})\to T(I\times I^{2n+q+2})\] be monomorphisms given by the matrix 

\[ F^{\delta}_{(t,x)}=\left( \begin{array}{cc}
 1 & 0\\
\partial_tb^{\delta}_{\sigma(t)}(x) & F^{\delta}_{\sigma(t)}(x)\\
 \end{array}\right)\]
 
 Which covers $b^{\delta}(t,x)=(t,b^{\delta}_{\sigma(t)}(x))$. So $F^{\delta}=db^{\delta}$ on $I\times (I^{2n+q}-I^{2n+q}_{\varepsilon(\delta)})$. Let $\chi^{\delta}:I^{2n+q+1}\to I$ be a smooth map such that $\chi^{\delta}=0$ on $I^{2n+q+1}-I^{2n+q+1}_{\varepsilon(\delta)}$ and $\chi^{\delta}=1$ on $I^{2n+q+1}_{\delta'}$, $\delta'<\delta$.

 Set $\Xi_{\tau}:I^{2n+q}\to I^{2n+q},\ \tau\in I$ as \[\Xi_{\tau}(x_1,...,x_{2n+q})=(x_1,...,x_{q-1},(1-\chi^{\delta})x_q-\chi^{\delta}(\eta\tau-x_q),x_{q+1},...,x_{2n+q})\] Where $\eta>0$ is so small such that $x_q-\eta \chi^{\delta}$ remains non negative. Now set $(F^{\delta}_{\tau})_{(t,x)}=F^{\delta}_{(t,\Xi_{\tau}(x))}$ which covers $b^{\delta}_{\tau}(t,x)=b^{\delta}(t,\Xi_{\tau}(x))$. Observe that 
 \begin{enumerate}
 \item $F^{\delta}_{\tau}=db^{\delta}=db^{\delta}_{\tau}$ on $(I-I_{\varepsilon(\delta)})\times(I^{q-1}- I^{q-1}_{\varepsilon(\delta)})\times I \times (I^{2n}-I^{2n}_{\varepsilon(\delta)})$\\
 \item $F^{\delta}_0=db^{\delta}=db^{\delta}_{\tau}$ on $I\times I^{q-1}\times [0,\varepsilon(\delta)]\times I^{2n}$ as $\chi^{\delta}=0$ on it.\\
 \item $F^{\delta}_1=db^{\delta}=db^{\delta}_{\tau}$ on $I\times I^{q-1}\times [1-\varepsilon(\delta),1]\times I^{2n}$ as $\chi^{\delta}=0$ on it.\\
 \end{enumerate}
 Moreover observe that $F^{\delta}_0$ is holonomic and for $\tau\in I_{\delta}$ \[\Sigma^{\delta}_{\tau}=\{F^{\delta}_{\tau}\ not\ \pitchfork\ to\ T\mathcal{L}\times \mathbb{R}^2\}\subset I^{2n+q+1}-I^{2n+q+1}_{\delta}\]
 
 Using the \ref{HAT} we can approximate $F^{\delta}_{\tau}$ on $Op(h^1_{\tau}(I\times I^{q-1}\times \{1/2\}\times I^{2n}))$ by $df^{\delta}_{\tau}$, where $f^{\delta}_{\tau}$ is a family of immersions defined on $Op(h^1_{\tau}(I\times I^{q-1}\times \{1/2\}\times I^{2n}))$. Observe that \ref{HAT} can not be applied directly as $F^{\delta}_1$ is not holonomic but this can be achieved by reparametrization of $F^{\delta}_{\tau}$ as follows. First on $\tau\in [0,1/2]$ define $F^{\delta}_{\tau}=F^{\delta}_{2\tau}$ and on $[1/2,1]$ define $F^{\delta}_{\tau}=F^{\delta}_{2-2\tau}$, so now we can apply \ref{HAT} but after applying it we will have to consider the resulting immersion only for parameter value on $[0,1/2]$ and then reparametrize accordingly.\\

  Now consider two smooth functions $\chi^i,\ i=1,2$ defined as follows\\
 $\chi^1:[0,1/2]\to [0,1]$, $\chi^1=0,\ on\ Op(0)$ and $\chi^1=1,\ on\ Op(1/2)$. $\chi^2:[1/2,1]\to [0,1]$, $\chi^2=0,\ on\ Op(1/2)$ and $\chi^2=1,\ on\ Op(1)$. Now define $g^{\delta}_{\tau}$ as follows 
 \[
\begin{array}{rcl}
g^{\delta}_{\tau} &=& b^{\delta}_0\circ g_{\chi^1(\tau)},\ \tau \in [0,1/2]\\
 &=& f^{\delta}_{\chi^2(\tau)}\circ h^1_{\chi^2(\tau)}\circ g_1,\ \tau \in [1/2,1]\\
  \end{array} 
 \]
 Where $g_s:I\times I^{2n+q}\to I\times I^{2n+q+2},\ s\in I$ is an isotopy of embeddings defined as follows

 
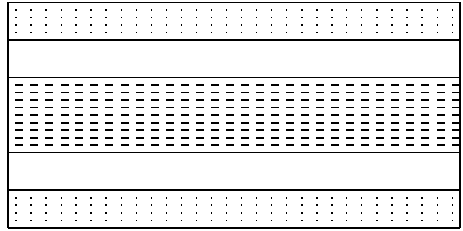
\begin{center}
\begin{picture}(300,150)(-100,5)\setlength{\unitlength}{1cm}
\linethickness{.075mm}

\multiput(-1,1.5)(6,0){2}
{\line(0,1){3}}

\multiput(-1,1.5)(0,3){2}
{\line(1,0){6}}
\multiput(-1,2)(0,2){2}
{\line(1,0){6}}
\multiput(-1,2.5)(0,1){2}
{\line(1,0){6}}
\put(.5,1){$I\times I^{2n+q}=g_0(I\times I^{2n+q})$}

\multiput(-.9,2.6)(.2,0){30}{\line(1,0){.09}}
\multiput(-.9,2.7)(.2,0){30}{\line(1,0){.09}}
\multiput(-.9,2.8)(.2,0){30}{\line(1,0){.09}}
\multiput(-.9,2.9)(.2,0){30}{\line(1,0){.09}}
\multiput(-.9,3)(.2,0){30}{\line(1,0){.09}}
\multiput(-.9,3.1)(.2,0){30}{\line(1,0){.09}}
\multiput(-.9,3.2)(.2,0){30}{\line(1,0){.09}}
\multiput(-.9,3.3)(.2,0){30}{\line(1,0){.09}}
\multiput(-.9,3.4)(.2,0){30}{\line(1,0){.09}}

\multiput(-.9,1.6)(.2,0){30}{\line(1,0){.00005}}
\multiput(-.9,1.7)(.2,0){30}{\line(1,0){.00005}}
\multiput(-.9,1.8)(.2,0){30}{\line(1,0){.00005}}
\multiput(-.9,1.9)(.2,0){30}{\line(1,0){.00005}}

\multiput(-.9,4.1)(.2,0){30}{\line(1,0){.00005}}
\multiput(-.9,4.2)(.2,0){30}{\line(1,0){.00005}}
\multiput(-.9,4.3)(.2,0){30}{\line(1,0){.00005}}
\multiput(-.9,4.4)(.2,0){30}{\line(1,0){.00005}}

\end{picture}\end{center}

 
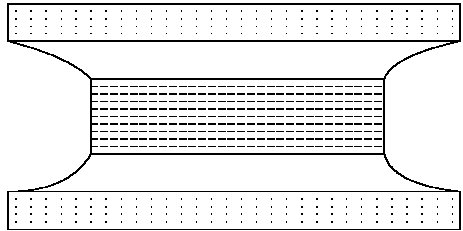
\begin{center}
\begin{picture}(300,150)(-100,5)\setlength{\unitlength}{1cm}
\linethickness{.075mm}

\multiput(-1,4)(6,0){2}
{\line(0,1){.5}}
\multiput(-1,1.5)(6,0){2}
{\line(0,1){.5}}
\multiput(.1,2.5)(3.9,0){2}
{\line(0,1){1}}

\multiput(-1,1.5)(0,3){2}
{\line(1,0){6}}
\multiput(-1,2)(0,2){2}
{\line(1,0){6}}
\multiput(.1,2.5)(0,1){2}
{\line(1,0){3.9}}
\put(1.7,1){$g_s(I\times I^{2n+q})$}

\multiput(.1,2.6)(.13,0){30}{\line(1,0){.09}}
\multiput(.1,2.7)(.13,0){30}{\line(1,0){.09}}
\multiput(.1,2.8)(.13,0){30}{\line(1,0){.09}}
\multiput(.1,2.9)(.13,0){30}{\line(1,0){.09}}
\multiput(.1,3)(.13,0){30}{\line(1,0){.09}}
\multiput(.1,3.1)(.13,0){30}{\line(1,0){.09}}
\multiput(.1,3.2)(.13,0){30}{\line(1,0){.09}}
\multiput(.1,3.3)(.13,0){30}{\line(1,0){.09}}
\multiput(.1,3.4)(.13,0){30}{\line(1,0){.09}}

\multiput(-.9,1.6)(.2,0){30}{\line(1,0){.00005}}
\multiput(-.9,1.7)(.2,0){30}{\line(1,0){.00005}}
\multiput(-.9,1.8)(.2,0){30}{\line(1,0){.00005}}
\multiput(-.9,1.9)(.2,0){30}{\line(1,0){.00005}}

\multiput(-.9,4.1)(.2,0){30}{\line(1,0){.00005}}
\multiput(-.9,4.2)(.2,0){30}{\line(1,0){.00005}}
\multiput(-.9,4.3)(.2,0){30}{\line(1,0){.00005}}
\multiput(-.9,4.4)(.2,0){30}{\line(1,0){.00005}}

\qbezier(-1,2)(-.2,2)(.1,2.5)
\qbezier(-1,4)(-.2,3.8)(.1,3.5)
\qbezier(4,2.5)(4.1,2.1)(5,2)
\qbezier(4,3.5)(4.1,3.8)(5,4)

\end{picture}\end{center}

 
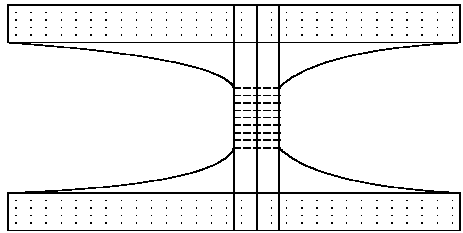
\begin{center}
\begin{picture}(300,150)(-100,5)\setlength{\unitlength}{1cm}
\linethickness{.075mm}

\multiput(2,1.5)(.6,0){2}
{\line(0,1){3}}
\multiput(2.308,1.5)(.6,0){1}
{\line(0,1){3}}

\multiput(-1,4)(6,0){2}
{\line(0,1){.5}}
\multiput(-1,1.5)(6,0){2}
{\line(0,1){.5}}

\multiput(-1,1.5)(0,3){2}
{\line(1,0){6}}
\multiput(-1,2)(0,2){2}
{\line(1,0){6}}
\put(1.7,1){$g_1(I\times I^{2n+q})$}

\multiput(2,2.6)(.13,0){5}{\line(1,0){.09}}
\multiput(2,2.7)(.13,0){5}{\line(1,0){.09}}
\multiput(2,2.8)(.13,0){5}{\line(1,0){.09}}
\multiput(2,2.9)(.13,0){5}{\line(1,0){.09}}
\multiput(2,3)(.13,0){5}{\line(1,0){.09}}
\multiput(2,3.1)(.13,0){5}{\line(1,0){.09}}
\multiput(2,3.2)(.13,0){5}{\line(1,0){.09}}
\multiput(2,3.3)(.13,0){5}{\line(1,0){.09}}
\multiput(2,3.4)(.13,0){5}{\line(1,0){.09}}

\multiput(-.9,1.6)(.2,0){30}{\line(1,0){.00005}}
\multiput(-.9,1.7)(.2,0){30}{\line(1,0){.00005}}
\multiput(-.9,1.8)(.2,0){30}{\line(1,0){.00005}}
\multiput(-.9,1.9)(.2,0){30}{\line(1,0){.00005}}

\multiput(-.9,4.1)(.2,0){30}{\line(1,0){.00005}}
\multiput(-.9,4.2)(.2,0){30}{\line(1,0){.00005}}
\multiput(-.9,4.3)(.2,0){30}{\line(1,0){.00005}}
\multiput(-.9,4.4)(.2,0){30}{\line(1,0){.00005}}

\qbezier(-1,2)(1.8,2.1)(2,2.6)
\qbezier(-1,4)(1.8,3.8)(2,3.4)
\qbezier(2.6,2.6)(3.1,2.1)(5,2)
\qbezier(2.6,3.4)(3.1,3.9)(5,4)

\end{picture}\end{center}

 $g_s=id,\ s\in I$ on $(I-I_{\varepsilon(\delta)/2})\times (I-I_{\varepsilon(\delta)/2})^{q-1}\times I\times (I-I_{\varepsilon(\delta)/2})^{2n}$. This is shown as shaded region at the top and bottom in the pictures above.\\
 
 Let $\bar{g}_s:I\to I$ be such that $\bar{g}_0=id$ and $\bar{g}_1(I)\subset Op(1/2)$. Then we set $g_s=id_{I_{\varepsilon(\delta)}}\times id_{I^{q-1}_{\varepsilon(\delta)}}\times \bar{g}_s \times id_{I^{2n}_{\varepsilon(\delta)}} $ on $I_{\varepsilon(\delta)}\times I^{q-1}_{\varepsilon(\delta)}\times I \times I^{2n}_{\varepsilon(\delta)}$. This is shown in the central shaded region in the above pictures. $g_s$ can also be arranged such that $g_1(\Sigma^{\delta}_t)$ does not intersect $I\times I^{q-1}\times \{1/2\}\times I^{2n}$.\\
 
 In the non-shaded region in the third picture i.e, in the picture of $g_1(I^{2n+q+1})$, \[f^{\delta}_0=b^{\delta}_0\ and\ h^1_0=id\] and hence $g^{\delta}_{\tau}$ is well defined. The next part of the proof is same as \cite{Wrinkle}, we include it here for completion. Although in the last part we need to do some work.\\
 
 Now observe that for $\tau \in I_{\delta}$, $\{g^{\delta}_{\tau}\ not\ \pitchfork\ to\ \mathcal{L}\times \mathbb{R}\}\subset (I^{2n+q+1}-I_{\delta}^{2n+q+1})$. \\
 
  For an integer $l>0$ take a function $\phi_l:I\to I$ such that 
 \[
 \begin{array}{rcl}
 \phi_l &=& 1,\ on\ I_{1/(8l)}\\
  &=& 0,\ outside\ I_{1/(16l)}\\ 
 \end{array}
 \]
which is increasing on $[1/(16l),1/(8l)]$ and decreasing on $[1-1/(8l),1-1/(16l)]$. Set \[\gamma_l(t)=t+\phi_l(t)Sin(2\pi lt),\ t\in I\] Let $J_i$ be the interval of length $9/(16l)$ centered at $(2i-1)/2l$. Observe that $\gamma_l$ is non-singular outside $\cup J_i$ and $(\gamma_l)_{J_i}$ is a wrinkle. Also \[\partial_t\gamma_l(t)\geq l,\ t\in I-\cup J_i\]
Let $\bar{\chi}^{\delta}:I^{2n+q+1}\to I$ be such that 
\[
\begin{array}{rcl}
\bar{\chi}^{\delta} &=& 0,\ near\ \partial(I^{2n+q+1})\\
  &=& 1,\ on\ I^{2n+q+1}_{\varepsilon(\delta)}
\end{array}
\]
Now we take $\delta=\delta(l)<<1/(16l)$. Set $\tilde{\gamma}_l(x)=(1-\bar{\chi}^{\delta}(x))x_q+\bar{\chi}^{\delta}(x)\gamma_l(x_q)$. Let $\lambda:I\to I$, be such that $\lambda(0)=0$, $\lambda(1)=1$ 
\begin{enumerate}
\item $\lambda=(2i-1)/2l,\ on\ J_i$\\
\item $0<\partial_t\lambda <3,\ on\ I-\cup J_i$\\
\end{enumerate}
Set $\bar{g}^{\delta}_{\tau}=g^{\delta}_{\lambda(\tau)},\ \tau\in I$. Now consider \[(t,x_1,...,x_{2n+q})\stackrel{\rho_l}{\to}\bar{g}^{\delta}_{x_q}(t,x_1,...,x_{q-1},\tilde{\gamma}_l(x),x_{q+1},...,x_{2n+q})\] Let $\theta$ be the function $\theta(t,x)=(t,x_1,...,x_{q-1},\tilde{\gamma}_l(x),x_{q+1},...,x_{2n+q})$. Then $\theta$ is a wrinkle map and as $\delta=\delta(l)<<1/(16l)$, the wrinkles of $\theta$ do not intersect $\{g^{\delta}_{\tau}\ not\ \pitchfork\ to\ \mathcal{L}\times \mathbb{R}\}$, for $\tau \in I_{\delta}$. On $I\times I^{q-1}\times J_i\times I^{2n}$, $\rho_l$ is of the form $\bar{g}^{\delta(l)}_i\circ \theta_i$, where $\theta_i=\theta_{\mid I\times I^{q-1}\times J_i\times I^{2n}}$. So using \ref{WrinkleKey} we can replace $\theta_i$ by another wrinkle map $\psi_i$ such that $\pi \circ \bar{g}^{\delta(l)}_i\circ \psi_i$ turns out to be a fibered wrinkle map, fibered over the first factor $I$. But observe that $\bar{g}^{\delta(l)}_i\circ \psi_i$ is not an immersion. So we need to regularize it. \\

In the proof of \ref{WrinkleKey} in \cite{Wrinkle} the coordinates has never been altered and kept fixed (Please see the Appendix \ref{Appendix} bellow). Moreover in the definition of wrinkle map $w_s$ and in the definition of special wrinkle (\cite{Wrinkle} which has been used in the proof of 
\ref{WrinkleKey} in \cite{Wrinkle}) the $y$ coordinates are kept fixed $x$ coordinates dropped and only the $z$ coordinate is changed and in the case of special wrinkles only the last coordinate is changed and the rest of the coordinates are kept fixed. \\

So from the proof of \ref{WrinkleKey} in \cite{Wrinkle} we may assume without loss of generality that for all $i$, $\pi \circ \bar{g}^{\delta(l)}_i\circ \psi_i$ has many wrinkles and near each wrinkle it  is of the form \[w_s(t,y,z,x)=(t,y,z^3+3(|(t,y)|^2-1)z-\Sigma_1^sx_i^2+\Sigma_{s+1}^{2n}x_i^2)\] and on this wrinkle the foliation is induced by $\pi$ and hence $\bar{g}^{\delta(l)}_i\circ \psi_i$ is of the form \[(t,y,z,x)\mapsto(t,y,z^3+3(|(t,y)|^2-1)z-\Sigma_1^sx_i^2+\Sigma_{s+1}^{2n}x_i^2,a_1(t,y,z,x),...,a_{2n+2}(t,y,z,x))\] Its derivative is given by the matrix
 \[ \left( \begin{array}{ccc}
 I_{q} & 0 & 0\\
* & 3(z^2+|(t,y)|^2-1) & (\pm 2x_i)_1^{2n}\\
* & (\partial_za_j)_1^{2n+2} & (\partial_{x_i}a_j)_{i=1,j=1}^{i=2n,j=2n+2}\\
 \end{array}\right)\]
 
In order to regularize it is enough to $C^1$-approximate $a_j$'s by $a'_j$'s. This can be done by $C^1$-perturbing $\psi_i$ only in those directions suitably.\\

 But we shall moreover want the $\partial_za'_{2n+1}\neq 0$ along $\{z^2+|(t,y)|^2-1=0\}$, where $a'_{2n+1}$ corresponds to the $\mathbb{R}$-factor of $\tilde{M}=M\times \mathbb{R}$. Moreover we also want $\partial_{y_1}a'_{2n+2}+1 \neq 0$.\\
 
 Now we need to convert non-transversality into tangency. Let us set $\varphi:I^{2n+q+2}\to [0,1]$ be a smooth function such that $\varphi=1$ outside a neighborhood of $D$, where $D$ is the disc which encloses $\{z^2+|(t,y)|^2-1=0\ and\ x=0\}$ and on $\{z^2+|(t,y)|^2-1=0\}$, $\varphi=0$ and $\partial_{y_1}\varphi=0$, moreover $\varphi+y_1\partial_{y_1}\varphi$ is non-vanishing outside $\{z^2+|(t,y)|^2-1=0\}$. We pictorially show the graph of $\varphi$ which shows its existence.\\
 
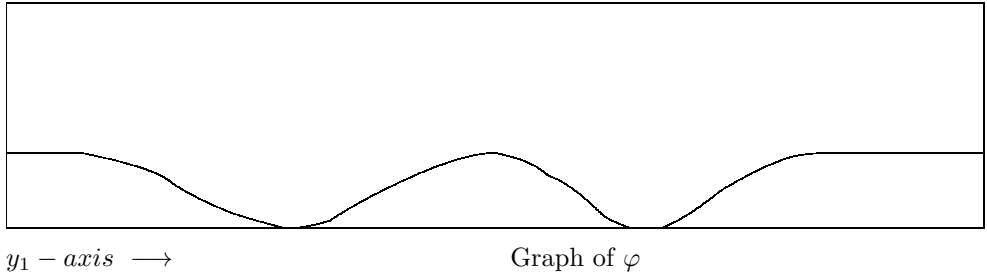
 \begin{center}
\begin{picture}(300,150)(-100,5)\setlength{\unitlength}{1cm}
\linethickness{.075mm}

\multiput(-5,1.5)(13,0){2}
{\line(0,1){3}}

\multiput(-5,1.5)(0,3){2}
{\line(1,0){13}}
\put(1.7,1){Graph of $\varphi$}
\put(-5,1){$y_1-axis\ \longrightarrow$}

\qbezier(-5,2.5)(-4,2.5)(-4,2.5)
\qbezier(-4,2.5)(-3,2.3)(-2.8,2.1)
\qbezier(-2.8,2.1)(-2.5,1.9)(-2,1.7)
\qbezier(-2,1.7)(-1.7,1.6)(-1.3,1.5)
\qbezier(-1.3,1.5)(-1,1.5)(-0.7,1.6)
\qbezier(-0.7,1.6)(-0.4,1.8)(0,2)
\qbezier(0,2)(1,2.5)(1.5,2.5)
\qbezier(1.5,2.5)(2,2.4)(2.2,2.2)
\qbezier(2.2,2.2)(2.5,2.1)(2.9,1.7)
\qbezier(2.9,1.7)(3,1.6)(3.3,1.5)
\qbezier(3.3,1.5)(3.5,1.5)(3.7,1.5)
\qbezier(3.7,1.5)(4,1.6)(4.5,2)
\qbezier(4.5,2)(5,2.3)(5.3,2.4)
\qbezier(5.3,2.4)(5.5,2.5)(6,2.5)
\qbezier(6,2.5)(7,2.5)(8,2.5)

\end{picture}\end{center}

Where the graph touches the horizontal bottom line, i.e, the $y_1$-axis represents the sphere $\partial D \cap \{y_1-axis\}$ where \[\partial D=\{z^2+|(t,y)|^2=1\ and\ x=0\}\]

  Now let $y=(y_1,...,y_q)$ in the above. Now replace the resulting map by \[(t,y,z,x)\mapsto(t,\varphi(t,y,z,x)y_1,y_2,...,y_q,z^3+3(|(t,y)|^2-1)z-\Sigma_1^sx_i^2\] \[+\Sigma_{s+1}^{2n}x_i^2,a'_1(t,y,z,x),...,a'_{2n+2}(t,y,z,x)+y_1-y_1\varphi(t,y,x,z))\] Where in the above the last component corresponds to the $\mathbb{R}$-component of $\tilde{M}\times \mathbb{R}$, i.e, the $w$-variable. Its derivative is given by 
 
 \[ \left( \begin{array}{ccccc}
 1 & 0 & 0 & 0 & 0\\
\partial_t(y_1\varphi) & \varphi+y_1 \partial_{y_1}\varphi & * & * & *\\
0 & 0 & I_{q-2} & 0 &0\\
* & * & * & 3(z^2+|(t,y)|^2-1) & (\pm 2x_i)_1^{2n}\\
* & * & * & (\partial_za'_j)_1^{2n+1} & (\partial_{x_i}a'_j)_{i=1,j=1}^{i=2n,j=2n+1}\\
*-\partial_t(y_1\varphi) & *+1-\partial_{y_1}(y_1\varphi) & *-\partial_{y_k}(y_1\varphi) & (\partial_za'_{2n+2}-\partial_z(y_1 \varphi)) & (\partial_{x_i}a'_{2n+2}-\partial_{x_i}(y_1 \varphi))_{i=1}^{i=2n}\\
 \end{array}\right)\]
 
 
 Now observe that the projections of the column vectors \[(0,*,0,3(z^2+|(t,y)|^2-1),(\partial_za'_j)_1^{2n+1},(\partial_za'_{2n+2}-\partial_z(y_1 \varphi)))^T\] and \[(0,*,0,(\pm 2x_i)_1^{2n},(\partial_{x_i}a'_j)_{i=1,j=1}^{i=2n,j=2n+1},(\partial_{x_i}a'_{2n+2}-\partial_{x_i}(y_1 \varphi))_{i=1}^{i=2n})^T\] onto $T\tilde{\mathcal{F}}\times \mathbb{R}$ are \[((\partial_za'_j)_1^{2n+1},(\partial_za'_{2n+2}-\partial_z(y_1 \varphi)))^T\] and \[((\partial_{x_i}a'_j)_{i=1,j=1}^{i=2n,j=2n+1},(\partial_{x_i}a'_{2n+2}-\partial_{x_i}(y_1 \varphi))_{i=1}^{i=2n})^T\] and their projection on $T\tilde{\mathcal{F}}$ are \[((\partial_za'_j)_1^{2n+1})^T\] and \[((\partial_{x_i}a'_j)_{i=1,j=1}^{i=2n,j=2n+1})^T\] Whose span was already close to $\mathbb{R}\times L_t$. Let us study the effect of the operation on the second column vector. Outside $Op(D)$ it is $(0,1,0,...,0,\partial_{y_1}a'_{2n+2})$ and on $\partial D$ it is \[(0,...,0,\partial_{y_1}a'_{2n+2}+1)\] Recall $\partial_{y_1}a'_{2n+2}+1 \neq 0$. So we have rotated the vector in order to make the $y_1$ component of the vector zero and last component of it to be non-zero along $\partial D$.\\ 
 
  Along $\{g^{\delta}_{\tau}\ not\ \pitchfork\ to\  \mathcal{L}\times \mathbb{R}\}\subset I^{2n+q+1}-I^{2n+q+1}_{\delta},\ \tau\in I_{\delta}$, we can apply the same technique as above. We rotate the $y$ and $z$ components simultaneously to make them tangent to $\mathcal{L}\times \mathbb{R}$.  This way we make $\{g^{\delta}_{\tau}\ not\ \pitchfork\ to\  \mathcal{L}\times \mathbb{R}\}$ to $\{g^{\delta}_{\tau}\ tangent\ to\  \mathcal{L}\times \mathbb{R}\}$.\\
  
  Note that if $q>2$, then along a possible intersection of the three sets \[\{\partial_{z}\ tangent\ to\ \mathcal{L}\times \mathbb{R}\}\cap_{i=1}^2\{\partial_{y_i}\ tangent\ to\ \mathcal{L}\times \mathbb{R}\}\] we can not transform $\{g^{\delta}_{\tau}\ not\ \pitchfork\ to\  \mathcal{L}\times \mathbb{R}\}$ to $\{g^{\delta}_{\tau}\ tangent\ to\  \mathcal{L}\times \mathbb{R}\}$, otherwise the rank will drop and it will no longer be regular.\\
 
  Let $\bar{\rho}_l$ be the regularized map, then $\bar{\rho}_l$ is of the form $\bar{\rho}_l(t,x)=(t,x(t))$, where $x(t)$ are functions of $t$. So the required family of immersions is given by \[f_t(x)=x(\sigma^{-1}(t)),\ t\in [0,1/2]\] with reparametrization. Clearly $f_t$ has the property $(1)\ and\ (2)$. Condition $(3)$ follows from the fact that for large $l$, $d_I\rho_l$ approximates $d_I\bar{g}^{\delta}_{\tau}$ on $I\times I^{q-1}\times (I-\cup_i J_i) \times I^{2n}$ and on $I^{2n+q+1}-I^{2n+q+1}_{\delta(l)}$ whose proof is same as in 2.3A of \cite{Wrinkle} and we refer the readers to \cite{Wrinkle}. As $\delta(l)$ depends on $l$ and $\varepsilon(\delta)$ depends on $\delta$, we are done.  
\end{proof}
\begin{remark}
In \ref{Main-1} if we drop the homotopy part then it seems to be possible to show the existence of the poisson structure $\pi_1$ even in the case where the codimension-$q$ is bigger that $2$. 
\end{remark}

\section{Appendix}
\label{Appendix}
 Here we define the function $\psi$ as in the conclusion of \ref{WrinkleKey}. The proof that $\hat{\beta}$ (below) is a wrinkle map is done in the proof of {\bf 2.2A} in \cite{Wrinkle}.\\

 Let $U,V$ be open domains in $\mathbb{R}^n$ and $\hat{\alpha}:U\to V$ be a special wrinkle. Let the sphere $\Sigma \subset U$ be the singular locus of $\hat{\alpha}$ and $D$ be dick bounded by $\Sigma$.\\

Consider a smooth even function $a_M:\mathbb{R}\to \mathbb{R}_+$ such that $a_M(x)=M^2x^2$ on $|x|<(1/3M)$ and on $|x|>(1/2)$, $a_M(x)=|x|$ moreover $\partial_xa_M(x)>1/2$ for $x\geq 1/3M$. Let $b_{N,M}$ be an odd, $2/N$-periodic smooth function which is equal to $a_M(2Nx+1)-1$ on $[-1/N,0]$. Define \[\beta(x_1,...,x_n)=\alpha(x_1,...,x_n)+\delta \rho(x_1,...,x_n)(b(x_1-1/2N)+...+b(x_{n-q}-1/2N))\] where $\delta>0$ is small real number and $\rho$ is a cut-off function on $U$ which equals $1$ on a neighborhood of $D$ and $0$ outside a compact subset of $U$.\\

Then for $\delta$ sufficiently small and $N$ and $M$ sufficiently large, $\psi=\hat{\beta}$ is the required function.

{\bf Acknowledgement:} I would like to thank the reviewer for his comments.\\

\end{document}